\documentclass[11pt]{article}

\usepackage{amsmath, amssymb, amsthm, enumerate, eurosym}
\usepackage{amssymb,dsfont}
\usepackage{latexsym}
\usepackage{graphicx}
\usepackage{bm}
\usepackage{amsfonts}
\usepackage{tikz}
\usepackage{ragged2e}
\usepackage{lmodern}
\usepackage[colorlinks=true,citecolor=blue, linkcolor=red, urlcolor=violet]{hyperref}
\usepackage{esint}
\usepackage{mathtools}
\usepackage{setspace}
\usepackage[authoryear, sort]{natbib}
\usepackage{float}
\usepackage{comment}

\newcommand{\diff}{\mathrm{d}}

\newcommand{\ind}[1]{\mathds{1}_{\{#1\}}}
\numberwithin{equation}{section}
\numberwithin{figure}{section}
\numberwithin{table}{section}

\def\MRV{\mathcal{MRV}}

\def\RV{\mathcal{RV}}

\def\bC{\mathbb{C}}

\def\bzero{\boldsymbol 0}
\def\bone{\boldsymbol 1}

\def\bx{\boldsymbol x}

\def\ba{\boldsymbol a}

\def\cG{\mathcal{G}}
\def\cV{\mathcal{V}}
\def\cE{\mathcal{E}}
\iftrue 
\usepackage{amsmath}
\usepackage{amsfonts}
\newcommand{\field}[1]{\mathbb{#1}}
\DeclareMathOperator{\PR}{\field{P}}             
\DeclareMathOperator{\E}{\field{E}}              
\def\N{\field{N}}                                
\def\R{\field{R}}                                
\def\F{\field{F}}                                
\def\M{\field{M}}                                
\newcommand{\convp}{\stackrel{p}{\longrightarrow}}
\else
\def\PR{\mathop{\rm I\kern -0.20em P}\nolimits}  
\def\E{\mathop{\rm I\kern -0.20em E}\nolimits}   
\def\N{\mathop{\rm I\kern -0.20em N}\nolimits}   
\def\R{\mathop{\rm I\kern -0.20em R}\nolimits}   
\def\F{\mathop{\rm I\kern -0.20em F}\nolimits}   
\fi

\theoremstyle{plain} 
\newtheorem{theorem}{Theorem}[section]
\newtheorem{lemma}[theorem]{Lemma}
\newtheorem{prop}[theorem]{Proposition}

\theoremstyle{definition} 

\newtheorem{defn}[theorem]{Definition}
\newtheorem{example}[theorem]{Example}
\newtheorem{remark}{Remark}

\newcommand{\tw}[1]{{\color{blue} #1}}

\begin{document}

\title{{Asymptotic Behavior of Common Connections in Sparse Random Networks}}

 \author{ {Bikramjit Das}\thanks{Engineering Systems and Design, Singapore University of Technology and Design, 8 Somapah Road, 487372 Singapore, \emph{E-mail:} {bikram@sutd.edu.sg}} 
 \and 
 {Tiandong Wang}\thanks{Department of Statistics, Texas A\&M University, TX77843, United States, \emph{E-mail:} {twang@stat.tamu.edu}}
 \and
{Gengling Dai}\thanks{Engineering Systems and Design, Singapore University of Technology and Design, 8 Somapah Road, 487372 Singapore, \emph{E-mail:} gengling\_dai@mymail.sutd.edu.sg}
}

\maketitle
\begin{abstract}

Random network models generated using sparse exchangeable graphs have provided a mechanism to study a wide variety of complex real-life networks. In particular, these models help with investigating power-law properties of degree distributions, number of edges, and other relevant network metrics which support the scale-free structure of networks. Previous work on such graphs imposes a marginal assumption of univariate regular variation  (e.g., power-law tail) on the bivariate generating graphex function. In this paper, we study sparse exchangeable graphs generated by graphex functions which are multivariate regularly varying. We also focus on a different metric for our study: the distribution of the number of common vertices (connections) shared by a pair of vertices. The number being high for a fixed  pair is an indicator of the original pair of vertices being connected. We find that the distribution of number of common connections are regularly varying as well, where the tail indices of regular variation are governed by the type of graphex function used. Our results are verified on simulated graphs by estimating the relevant tail index parameters.

\end{abstract}



\section{Introduction}

Modeling social, economic, and biological networks has become a principal area of interest for data scientists in the recent decades. The degree distribution provides an idea about the structure of the network, e.g., a power-law behavior here indicating the so-called ``scale-free" structure; see \cite{voitalov_etal:2019} for a recent discussion. Another quantity of potential interest to network scientists, especially in the context of social networking platforms like  Facebook, Instagram, LinkedIn, and Twitter, is the number of common connections or common friends between two vertices. This can be thought of as a generalization of the degree of a vertex. Network recommendation systems use this number as one of the metrics to suggest potential new vertices (friends) to platform users \citep{guptaetal:2013}.  The connection between \emph{link prediction} and the size of common connections has been explored in the literature; see \cite{liben-nowell:2007,newman:2001,barabasi:2002} for further discussions.
 
Although of interest to network scientists, only few theoretical results exist in the literature characterizing the asymptotic distribution of common connections for growing network models. While studying clustering coefficients for random intersection graph, \cite{bloznelis:kurauskus:2016} find the asymptotic distribution of common connections for a randomly selected pair of vertices. In addition, for a certain class of linear preferential attachment models, \cite{das:ghosh:2021} identify the growth rate for common connections between two fixed vertices in the graph. In our paper, we study common connections in the context of the erstwhile popular and flexible framework of exchangeable random graph models called \emph{graphex processes}.

Graphex models, or vertex-exchangeable random graphs have been used extensively to study statistical network models that are dense (\cite{hoover:1979}, \cite{aldous:1981}, \cite{lovasz:szegedy:2006}, \cite{diaconis:janson:2008}). Recent work has successfully extended the classical graphex framework to sparse networks, which mimics many real-world networks; see \cite{caron:fox:2017,borgs:2019}. In particular, they point out that exchangeable random measures admit a representation theorem due to \cite{kallenberg:1990}, giving a general construction for such graph models. Moreover, \cite{veitch:roy:2015} and \cite{borgs:chayes:cohn:holden:2018} show how such construction naturally generalizes the dense exchangeable graphex framework to the sparse regime, and analyze some of the properties of the associated class of random graphs, called graphex processes; also see for instance, \cite{janson:2016, janson:2017}, \cite{veitch:roy:2019} and \cite{borgs:2019}. Furthermore, \cite{herlau_etal:2016} and \cite{todeschini_etal:2020} have developed sparse graph models with (overlapping) community structure within the graphex framework. 

 In this paper, we work under the framework of sparse vertex-exchangeable random graphs generated by graphex functions which are \emph{multivariate regularly varying}, a modest generalization of power-law tailed functions. The regular variation of the graphex function creates a sparse graph, and we show that the distribution of common connections in the generated graphs is also regularly varying, i.e., approximately power-law tailed. Moreover, we observe that depending on the type of the graphex function, i.e., whether it is a product of univariate functions (called {\it{separable}}), or not (called {\it{non-separable}}), we get different asymptotic behaviors of common connections in the random graph. We estimate necessary parameters in simulated random graphs to numerically verify and support the results obtained.

The remaining of the paper proceeds as follows. In Section \ref{sec:randomgraphs}, we introduce the random graph framework, as well as necessary definitions and assumptions of multivariate regularly varying functions in this context. We also characterize effects of the regular variation assumption on graphex functions to particular marginals of the graphex functions. The main results on the asymptotic distribution of  common connections are collected in Section~\ref{sec:res}. In addition, Section~\ref{sec:sim} presents results from simulation studies, where we provide typical examples of sparse random networks and verify our results of Section~\ref{sec:res}.  Finally, concluding remarks are given in Section~\ref{sec:conclusion}.


\section{Random graphs, regular variation and graphex functions}\label{sec:randomgraphs}
In this section, we create our framework for network models via graphex processes, discuss multivariate regular variation, and derive regular variation properties for functionals of graphex processes.

\subsection{Random graph models}\label{subsec:sparsegraphs}
We now present a framework which models sparse random graphs using (infinite) Poisson point processes on $\R_+^2=[0,\infty)^2$. Finite graphs are obtained by appropriately truncating the support of the point process. The model representation is based on the work of \cite{caron:fox:2017,veitch:roy:2015,borgs:chayes:cohn:holden:2018,caron_etal:2020}. Here we follow their construction and notation. Such representation is also related to Kallenberg exchangeable graphs \citep{kallenberg:1990}, and can be characterized by a symmetric measurable function, $W:[0,\infty)^2\mapsto[0,1]$, often called a \emph{graphex function}. A variety of sparse random graphs can be parametrized using the graphex function $W$ (we ignore terms corresponding to isolated edges and stars).

Let $\Pi=(\theta_i,\eta_i)_{i=1,2,...}$ be a unit-rate Poisson process on $\R_+^2$.  In this representation, vertices are embedded at some location $\theta_i\in \R_+$, and there is a latent variable $\eta_i\in(0,\infty)$ which along with the graphex function $W$ determines the edges of the graph.  For this paper, we will define a graphex function as follows.
\begin{defn}\label{ass:W}
A symmetric measurable function $W:[0,\infty)^2\mapsto[0,1]$ is a \emph{graphex function} if it satisfies the following criteria:
\begin{enumerate}[(a)]
\item $0<\overline{W}=\int_0^\infty \int_0^\infty  W(x,y) \,\diff x \diff y<\infty$;
\item $\int_0^\infty W(x,x) \, \diff x<\infty$;
\item $\lim\limits_{x\to0} W(x,x)$ and $\lim\limits_{x\to\infty} W(x,x)$ both exist.
\end{enumerate}
\end{defn}
Note that we have $\lim_{x\to0} W(x,x)=0$, 
since $\overline{W}<\infty$.  
Other necessary regularity conditions on $W$ will be imposed later. Edges of the graph induced by $\Pi$ and $W$ are given by a point process:
  \begin{align*}
  Z= \sum_{i,j} Z_{ij} \delta_{(\theta_i,\theta_j)}
  \end{align*}
  where $Z_{ij}=Z_{ji}$ is a binary variable which takes the value 1 if there is an edge between $\theta_i$ and $\theta_j$ and 0 otherwise.
   Given the graphex function $W$ and $i\le j$ , we have
  \begin{align*}
  Z_{ij}|\Pi,W \sim \text{Bernoulli} (W(\eta_i,\eta_j))
  \end{align*}
and $\{Z_{ij}: i\le j\}$ are independent random variables. We obtain a finite size random graph family $(\mathcal{G}_t)_{t\ge0}$, often called a graphex process, by restricting  the process $\sum_{i,j} Z_{ij} \delta_{(\theta_i,\theta_j)}$ to  $[0,t]^2$. Here $\mathcal{G}_t = (\cE_t,\cV_t)$ denotes a graph of size $t\ge 0$ with edge set $\cE_t$ and vertex set $\cV_t$ given by 
 \begin{align*}
 \cE_t &= \{ \theta_i | \theta_i \le t \text{ and } Z_{ik}=1 \text{ for some $k$ with } \theta_k\le t\},\\
  \cV_t &= \{ (\theta_i,\theta_j) | \theta_i, \theta_j \le t \text{ and } Z_{ij}=1\}.
 \end{align*}

The following quantities are of interest for studying properties of graphex processes: the univariate and bivariate graphex marginal functions $\mu_{1}:\R_+\to\R_+$ and $\mu_{2}:\R_+^2\to\R_+$ are given by
\begin{align}
\mu_1(x) & = \int_0^\infty W(x,y) \diff y, \label{eq:mu}\\
\mu_2(x,y) & = \int_0^\infty W(x,z) W(y,z) \diff z. \label{eq:mu2}
\end{align}
The integrability of $W$ in $\R_+^2$ guarantees that both  $\mu_1$ and $\mu_2$ are well-defined and finite. We can check that $\mu_1(\eta_i)$ is proportional to  the average number of vertices that a vertex $\theta_i$ is connected to, whereas $\mu_2(\eta_i,\eta_j)$ is proportional to the average number of common connections (vertices) between the two vertices at $\theta_i$ and $\theta_j$. We also further define the $d$-variate graphex marginal for $d\ge 3$ as:
\begin{align}\label{eq:mud}
\mu_d(x_1,\ldots, x_d) := \int_0^\infty\prod_{i=1}^d W(x_i,z)\diff z, \qquad x_i\neq x_j,\quad i\neq j.
\end{align}

\subsection{Regular variation}\label{subsec:regvar}
Studies involving the Kallenberg exchangeable \emph{sparse} graphs have been occasionally carried out under the assumption that the univariate graphex marginal $\mu_1(x)$  has a power-law-like tail (regularly varying) as $x\to\infty$; see \cite{naulet_etal:2021,caron_etal:2020}. Instead of imposing a condition on $\mu_1$, we start by imposing important conditions on the generating graphex function, $W$, which  
is the direct input for the generation of the underlying network.
It turns out that the regularity conditions on $W$ are also closely connected to 
 the theory of multivariate regular variation \citep{bingham:goldie:teugels:1989,resnick:2007},
 which we now outline. 

A function $f:(0,\infty) \to (0,\infty)$ is regularly varying (at infinity) if $\lim_{t\to\infty} f(tx)/f(t) = x^{\beta}$ for $x>0$ and $\beta\in\R$. Here $\beta$ is the tail index or index of regular variation and we write $f\in \RV_{\beta}$.
In dimensions $d>1$ this can be extended to multivariate regular variation, cf. \cite{stam:1977,dehaan:resnick:1987}: suppose $\bC\subset \R_+^d$ is a cone, that is, $\bx\in \bC$ if and only if $t\bx\in \bC$ for all $t>0$. A function $f:\bC\to(0,\infty)$ is (multivariate) regularly varying with limit function $\lambda$ with $\lambda(\bx)>0$ for $\bx\in\bC$ and tail index $\beta$: if there exists a function $g:(0,\infty)\to (0,\infty)$ with $g\in\RV_{\beta}$ such that
\begin{align}\label{eq:mrvg}
\lim_{t\to\infty} \frac{f(t\bx)}{g(t)} &= \lambda(\bx), \bx \in \bC,
\end{align}
or, equivalently: if there exists a function $b:(0,\infty)\to (0,\infty)$ with $b\in\RV_{-1/\beta}$
\begin{align}\label{eq:mrvb}
\lim_{t\to\infty} t{f(b(t)\bx)} &= \lambda(\bx), \bx \in \bC.
\end{align}
As a consequence, we have the homogeneity property that $\lambda(t\bx) = t^{\beta} \lambda(\bx), \bx\in \bC, t>0$. Moreover, we can obtain \eqref{eq:mrvb} from \eqref{eq:mrvg} by defining $b(t) = g^{\leftarrow}(1/t)$ where $f^{\leftarrow}$ denotes the generalized left inverse of a monotone function; since without loss of generality, $g$ can be assumed to be monotone.  We write $f\in \MRV(\beta,g,\lambda,\bC)$ when following \eqref{eq:mrvg} or $f\in \MRV(\beta,b,\lambda,\bC)$ when following \eqref{eq:mrvb}; one or more of the parameters are also often dropped  for convenience. Both $b$ and $g$ are referred to as \emph{scaling} functions.

\subsection{Tail behavior of the univariate and bivariate graphex marginals}\label{subsec:regvarmargin}

The following result characterizes the asymptotic behaviour of $\mu_1$ and $\mu_2$ at infinity, given that $W$ is multivariate regularly varying. The proof uses ideas from  \cite[Theorem 2.1]{dehaan:resnick:1987}. A uniformity condition in addition to regular variation is imposed to guarantee regular variation of the integrals $\mu_1,\mu_2$, see  \cite[Section 2]{dehaan:resnick:1987} for further discussions. Let $\|\cdot\|: (x,y) \to \|(x,y)\|$ denote a norm on $\R^2$. 
 \begin{theorem}\label{thm:W_uni}
Suppose the graphex function $W:\R_{+}^{2}\to [0,1]$ is regularly varying on $\R^2_+\setminus\{\bzero\}$ with limit function $\omega$ and scaling function $h\in \RV_{-\alpha}$ where $\alpha>1$. Moreover assume that $\omega$ is bounded on  the unit sphere $\aleph :=\{(x,y) \in \R_+^2\backslash\{\mathbf{0}\} :\|(x,y)\| = 1\}$ and satisfies the uniformity condition
\begin{equation}
\label{W_uni}
\lim_{t\rightarrow\infty} \sup_{(x,y)\in \aleph} \Bigg|\frac{W(tx,ty)}{h(t)}-\omega(x, y)\Bigg| = 0.
\end{equation}
Then the following holds:
 \begin{enumerate}[(i)]
 \item The univariate graphex marginal $\mu_1$ is regularly varying with tail index $-(\alpha-1)$ and
 \begin{equation}
\label{mu1_uni}
\lim_{t\rightarrow\infty} \frac{\mu_1(tx)}{th(t)}=\int_0^\infty\omega(x,y) \,\diff y\quad\quad x>0.
 \end{equation}
 \item The bivariate graphex marginal $\mu_2$ is regularly varying on $(0,\infty)^2$ with tail index $-(2\alpha-1)$  and
  \begin{equation}
\label{mu2_bi}
\lim_{t\rightarrow\infty} \frac{\mu_2(tx,ty)}{th^2(t)}=\int_0^\infty\omega(x,z)\omega (y,z) \,\diff z\quad\quad (x,y) \in (0,\infty)^2.
 \end{equation}
\item The convergence in \eqref{mu2_bi} is uniform on sets bounded away from the axes, i.e., for any fixed $\delta>0$:
 \begin{align}\label{mu2_bi_uc}
 \lim_{t\rightarrow\infty} \sup_{x\wedge y \ge \delta} \Bigg|\frac{\mu_2(tx,ty)}{th^2(t)}-\int_0^\infty\omega(x,z)\omega (y,z) \; \diff z\Bigg| = 0.
 \end{align}
%
\end{enumerate} 

\end{theorem}

\begin{proof}
First, using a change of variable $v=tz$ we have
\begin{align}
\frac{\mu_1(tx)}{th(t)} & = \int_0^\infty \frac{W(tx,z)}{th(t)} \, \diff z = \int_0^\infty \frac{W(tx,tv)}{h(t)} \, \diff v, \quad\quad x>0, \label{eq:mu1con}\\
\frac{\mu_2(tx,ty)}{th^2(t)} & = \int_0^\infty \frac{W(tx,z)W(ty,z)}{th^2(t)} \, \diff z = \int_0^\infty \frac{W(tx,tv)}{h(t)}\frac{W(ty,tv)}{h(t)} \, \diff v \label{eq:mu2con}, \quad\quad (x,y)\in (0,\infty)^2.
\end{align}
By our assumptions, the integrands in \eqref{eq:mu1con} and \eqref{eq:mu2con} are non-negative and converge to $\omega(x,v)$ and $\omega(x,v)\omega(y,v)$ respectively as $t\to \infty$. Hence if we show that the integrands are bounded by an integrable function and the limit are integrable, then \eqref{mu1_uni} and \eqref{mu2_bi} holds using dominated convergence.  Since all norms in $\R^{2}$ are equivalent,  without loss of generality let us fix a particular norm in $\R^2$ given by $\|(x,y)\| = |x|\vee |y|$. 

 \begin{enumerate}[(i)]
 \item For fixed $x>0$, and any $v\ge 0$,  with $t>0$,
\begin{align*}
\frac{W(tx,tv)}{h(t)} = \frac{W(t\|(x,v)\|. \frac{x,v}{\|(x,v)\|})}{h(t\|(x,v)\|)} \frac{h(t\|(x,v)\|)}{h(t)}.
\end{align*}

Since $\|(x,v)\| \ge x>0$, and  $(x,v)/\|(x,v)\|\in \aleph$, given a fixed $\eta>0$, we obtain from \eqref{W_uni} that for $t>t_{0}\equiv t_{0}(x,\eta)$,
\begin{align}\label{eq:unifbd}
\frac{W\left(t\|(x,v)\|\frac{(x,v)}{\|(x,v)\|}\right)}{h(t\|(x,v)\|)} \le \sup_{\ba \in \aleph} \omega(\ba) +\eta  <\infty,
\end{align}
since $\omega$ is bounded on $\aleph$.
Using Potter's bounds \citep{resnick:2007}, for a large enough $t$, $h(t\|(x,v)\|)/h(t)$ is bounded by $c\|(x,v)\|^{-\alpha+\rho}$ for a constant $c>0$ and  $\rho$ such that $0<\rho<\alpha-1$. Hence,
\begin{align}
\frac{W(tx,tv)}{h(t)}  &\le C \|(x,v)\|^{-\alpha+\rho} , \label{eq:Wbd}\\
\intertext{for some constant $C>0$, and}
\int_0^{\infty}C \|(x,v)\|^{-\alpha+\rho} \, \diff v & = C x^{-\alpha+\rho} \int_0^{x} \diff v + C\int_x^{\infty}u^{-\alpha+\rho} \, \diff v <\infty,  \label{eq:Wbdint}
\end{align}
which shows that the integrand in \eqref{eq:mu1con} is bounded by an integrable function. By homogeneity property we also have $\omega(t(x,y)) = t^{-\alpha}\omega(x,y)$ for $(x,y)\in ,\R^2_+\setminus \{\bzero\}$ and hence
\begin{align}\label{eq:intomegafinite}
\int_0^{\infty} \omega(x,y) \;\diff y & = \int_0^{\infty} \|(x,y)\|^{-\alpha}\omega\left(\frac{(x,y)}{\|(x,y)\|}\right) \;\diff y
                                                                  \le \sup_{\ba \in \aleph} \omega(\ba)  \int_0^{\infty} (x\vee y)^{-\alpha}  \;\diff y <\infty,
 \end{align}
since $\omega$ is bounded in $\aleph$ and $\alpha>1$. Hence \eqref{mu1_uni} is a consequence of $W\in \MRV(-\alpha,h,\omega)$, \eqref{eq:Wbd}, \eqref{eq:Wbdint}, \eqref{eq:intomegafinite} and the dominated convergence theorem. Therefore $\mu_{1}\in \RV_{{-(\alpha-1)}}$.
 
\item Using the bounds in \eqref{eq:Wbd}, \eqref{eq:Wbdint} (and finding similar bounds for $W(ty,tv)/h(t)$), we can conclude that for any $(x,y) \in (0,\infty)^2$, the integrand in \eqref{eq:mu2con} is bounded by an integrable function. To see that the limit  (as $t\to\infty$) of the integrand in\eqref{eq:mu2con} is integrable, observe that
\begin{align*}
\int\limits_0^{\infty} \omega(x,z)\omega(y,z) \;\diff z & = \int\limits_0^{\infty} \|(x,z)\|^{-\alpha} \omega\left(\frac{(x,z)}{\|(x,z)\|}\right)\|(y,z)\|^{-\alpha} \omega\left(\frac{(y,z)}{\|(y,z)\|}\right) \;\diff z\\
& \le  \left( \sup_{\ba \in \aleph} \omega(\ba) \right)^2  \int\limits_0^{\infty} (x\vee z)^{-\alpha}  (y\vee z)^{-\alpha}  <\infty,
\end{align*}
since $\omega$ is bounded in $\aleph$ and $\alpha>1$.  Then \eqref{eq:mu2con} follows from the dominated convergence theorem, and we have $$\mu_{2}\in \MRV\left(-(2\alpha-1),th^{2}(t), \int_{0}^{\infty} \omega(x,z)\omega(y,z) \,\diff z,(0,\infty)^{2}\right).$$
\item From \citet[Theorem 2.1]{dehaan:resnick:1987} we know that the uniformity condition \eqref{W_uni} also implies that for any $\delta >0$,
\begin{align}\label{W_uni_alt}
\lim_{t\rightarrow\infty} \sup_{\|(x,y)\|>\delta} \Bigg|\frac{W(tx,ty)}{h(t)}-\omega(x, y)\Bigg| = 0.
\end{align}
Now for fixed $\delta>0$ and $x \wedge y >\delta$, 
 \begin{align*}
  \Bigg|&\frac{\mu_2(tx,ty)}{th^2(t)}-\int_0^\infty\omega(x,z)\omega (y,z)\Bigg|  \le \int_0^\infty  \Bigg| \frac{W(tx,tz)W(ty,tz)}{h^2(t)} -\omega(x,z)\omega(y,z)\Bigg|\; \diff z\\
 & \le \int_0^\infty   \frac{W(tx,tz)}{h(t)}\Bigg| \frac{W(ty,tz)}{h(t)} -\omega(y,z)\Bigg|\; \diff z +  \int_0^\infty  \omega(y,z)\Bigg| \frac{W(tx,tz)}{h(t)} -\omega(x,z)\Bigg|\; \diff z.
 \end{align*}
 Since $\{(x,y)\in \R^2_+: x\wedge y>\delta\}\subset\{(x,y)\in \R^2_+: x\vee y = \|(x,y)\|>\delta\}$, then by \eqref{W_uni_alt}, given $\epsilon>0$, there exists $t_0$ such that for
$t> t_0$, both $\big| \frac{W(ty,tz)}{h(t)} -\omega(y,z)\big|<\epsilon$ and $\big| \frac{W(tx,tz)}{h(t)} -\omega(x,z)\big|<\epsilon$ for any $(x,y)$ with $x\wedge y>\delta$. Therefore for $t>t_{0}$,
 \begin{align*}
  \sup_{x\wedge y>\delta}\Bigg|\frac{\mu_2(tx,ty)}{th^2(t)}-\int_0^\infty\omega(x,z)\omega (y,z)\Bigg| & \le \epsilon   \sup_{x>\delta} \int_0^\infty   \frac{W(tx,tz)}{h(t)} \,\diff z +  \epsilon   \sup_{ y>\delta}\int_0^\infty   \omega(y,z)\; \diff z\\
 &  = \epsilon  \sup_{x>\delta} A_{x,t}+\epsilon  \sup_{y>\delta} B_y \quad \text{(say)}.
  \end{align*}
  From \eqref{eq:Wbd} and \eqref{eq:Wbdint}, since $x>\delta>0$, for large enough $t$, $\sup\limits_{x>\delta} A_{x,t} \le C_1$ for some constant $C_1>0$. Also using\eqref{eq:intomegafinite}, we can check that $\sup\limits_{y>\delta} B_{y} \le C_2$ for some constant $C_2>0.$ Therefore
   \begin{align*}
 \lim_{t\rightarrow\infty} \sup_{x\wedge y \ge \delta} \Bigg|\frac{\mu_2(tx,ty)}{th^2(t)}-\int_0^\infty\omega(x,z)\omega (y,z) \; \diff z\Bigg| \le \epsilon(C_1+C_2)
 \end{align*}  
  and since $\epsilon>0$ is arbitrary, we have \eqref{mu2_bi_uc}.
\end{enumerate}
\end{proof}

 \begin{example}\label{ex:powerlaw}
 Consider the graphex function $W:\R_+^2\to(0,\infty)$ with $\alpha>1$:
 \begin{align}\label{fn:Walpha}
 W(x,y) = \frac{1}{1+x^{\alpha}+y^{\alpha}}.
 \end{align}
 Here $$
W\in \MRV\left(-\alpha,h(t)=t^{-\alpha},\omega(x,y) = \frac{1}{x^{\alpha}+y^{\alpha}},(0,\infty)^2\right),
$$ 
and we can check that the uniformity condition \eqref{W_uni} holds. Therefore the univariate marginal $\mu_1\in \RV_{-(\alpha-1)}$ and for $x>0$,
 \begin{align*}
 \lim_{t\rightarrow\infty} \frac{\mu_1(tx)}{th(t)}=\int_0^\infty\omega(x,y) \,\diff y=  x^{-\alpha+1} \int_0^{\infty} \frac{1}{1+z^{\alpha}} \diff z = \frac{\pi}{\alpha} cosec(\pi/\alpha)\; x^{-\alpha+1}.
 \end{align*}
 See \citet[\S 3.241.2]{gradshteyn:ryzhik:2007}. Moreover, the bivariate marginal satisfies 
 $$\mu_2\in \MRV(-(2\alpha-1), th^2(t), \nu(\cdot), (0,\infty)^{2}),$$ 
 where  $\nu$ is hard to compute in closed form in general. For $\alpha=2$, we can compute $\nu$ in closed form, which is given by
 \begin{align*}
 \nu(x,y)=\lim_{t\rightarrow\infty} \frac{\mu_2(tx,ty)}{th^{2}(t)} & =\int_0^\infty \frac{1}{x^{2}+z^{2}}\frac{1}{y^{2}+z^{2}} \,\diff z = \frac{\pi}{2x^2y+2xy^2}, \quad\quad (x,y)\in (0,\infty)^{2}.\end{align*}
  \end{example}

\begin{remark}
A classical  example of a multivariate regularly varying $W$ is when it is \emph{separable}, i.e., $W(x,y)=U(x)U(y)$  with $U\in \RV_{-\alpha}$ for some $\alpha>1$. Unfortunately, the uniformity condition \eqref{W_uni} fails to hold in this case and we cannot apply Theorem \ref{thm:W_uni}. Nevertheless, we can still ascertain the tail behavior in this case easily because of the separable structure.
\end{remark}

\begin{lemma}{\label{lem:Vseparable}}
Let $W:\R_+^2\rightarrow [0,1]$ be a graphex function such that $W(x,y)=U(x)U(y)$ where $U:\R_{+}\to(0,\infty)$ satisfies $U\in \RV_{-\alpha}$ for some $\alpha>1$.
Then the following holds:
\begin{enumerate}[(i)]
\item $W \in \MRV(-2\alpha, U^{2}(t), \omega(x,y),(0,\infty)^{2})$ where
$$\lim_{t\to\infty} \frac{W(tx,ty)}{U^{2}(t)} = x^{-\alpha}y^{-\alpha}=: \omega(x,y), \quad (x,y)\in (0,\infty)^{2},$$
\item $\mu_{1} \in \RV_{-\alpha}$ where
$$\lim_{t\to\infty} \frac{\mu_{1}(tx)}{U(t)} = x^{-\alpha} \int_{0}^{\infty} U(z) \,\diff z, \quad x>0,$$
\item $\mu_{2} \in \MRV(-2\alpha, U^{2}(t), C\omega(x,y),(0,\infty)^{2})$ where $C=\int_{0}^{\infty} U^{2}(z)\, \diff z$.
\end{enumerate}
\end{lemma}

\begin{proof} 
The proof is an easy consequence of $U\in \RV_{-\alpha}$ and is omitted here.
\end{proof}

\begin{example}\label{ex:prodpowerlaw}
 Consider the graphex function $W:\R_+^2\to(0,\infty)$ with $\alpha>1$:
 \begin{align*}
 W(x,y) =U(x)U(y) := \frac{1}{1+x^{\alpha}} \frac{1}{1+y^{\alpha}}.
 \end{align*}
Clearly, 
$$
W\in \MRV\left(-2\alpha,U^{2}(t),\omega(x,y) = \frac{1}{x^{\alpha}}\frac1{y^{\alpha}},(0,\infty)^2\right).
$$  
Using Example \ref{ex:powerlaw}, the univariate marginal $\mu_1\in \RV_{-\alpha}$ and for $x>0$,
 \begin{align*}
 \lim_{t\rightarrow\infty} \frac{\mu_1(tx)}{U(t)}=x^{-\alpha}\int_0^\infty U(z)\,\diff z=   x^{-\alpha}\int_0^\infty \frac{1}{1+z^{\alpha}}\,\diff z = \frac{\pi}{\alpha}cosec(\pi/\alpha)\; x^{-\alpha}.
 \end{align*}
Moreover, $\mu_2\in \MRV(-2\alpha, U^2(t), C\omega(x,y), (0,\infty)^{2})$, where  $C=((\pi/\alpha)cosec(\pi/\alpha))^{2}.$
 \end{example}


%
%
%


\section{Common connections}\label{sec:res}

In this section, our goal is to understand the behavior of the number of common connections between two vertices, first for a fixed pair of vertices and eventually for a randomly chosen pair in a graph driven by the (Kallenberg exchangeable) graphex process as described in Section \ref{sec:randomgraphs}. 

\subsection{Distribution of common connections in $\cG_t$}
Recall that from the unit-rate Poisson process $\Pi = (\theta_i,\eta_i)_{i=1,2, \ldots}$, the finite-size graphex $\cG_t$ is created by restricting  $\sum_{i,j}Z_{ij}\delta_{(\theta_i,\theta_j)}$ to $[0,t]^2$, where $Z_{ij} \sim \text{Bernoulli} (W(\eta_i,\eta_j))$, and $Z_{ij}=Z_{ji}$. The first co-ordinate in $\Pi$ does not provide any relevant information other than picking the plausible vertices, so for ease of computation, we project $\Pi$ to its second co-ordinate as  
$$\Pi_t^{\eta}=\{\eta_i|(\theta_i,\eta_i)\in\Pi_t\},$$ 
so that $\Pi_t^{\eta}$ is a one-dimensional Poisson process with rate $t$. 

Let $I$ be an index set. Then
following \cite[Section 5]{veitch:roy:2015}, for a locally finite, simple sequence $\phi = (x_i)_{i\in I}$, and a sequence of values $\{u_{ij}\}\in [0,1]$ with $u_{ij}=u_{ji}, i,j \in I$, we define the \emph{2nd order common connections function} or \emph{2-c-degree function} at (vertices) $x_i, x_j \in \phi$ to be
 \begin{align}
 C(x_i,x_j, \phi, (u_{ij})) & := \sum\limits_{x_k\in \Pi_t^{\eta} \setminus \{x_i,x_j\}} \bone_{\{W(x_i,x_k) \ge u_{ik}\}} \bone_{\{W(x_j,x_k)>u_{jk}\}}. \label{eq:2cd}
  \end{align}
With this definition, for the graphex $\cG_t$ realized at $Z_{ij}=z_{ij}$ (which are realizations of Unif$(0,1)$ random variables), the 2-c-degree function at $(\theta_1,\eta_1), (\theta_2,\eta_2) \in \Pi_t$ is 
\begin{align} 
 C(\eta_1, \eta_2, \Pi_t^{\eta}, (z_{ij})) \label{eq:2cdforGt}.
\end{align}
Taking $\eta_1,\eta_2$ to be the proxy for the points $(\theta_1,\eta_1), (\theta_2,\eta_2) \in \Pi_t$, we refer to $C(\eta_1, \eta_2, \Pi_t^{\eta}, (z_{ij}))$ as the number of common friends/connections of $\eta_1,\eta_2$ in $\cG_t$. We can analogously  define a \emph{k-c-degree} function, for the number of common connections between $k\ge 2$ vertices as well.

We now explore how $C(\eta_1, \eta_2, \Pi_t^{\eta}, (Z_{ij}))$ behaves for fixed $t$, and later
in Section~\ref{subsec:asy}, we study the asymptotic behavior of common friends for a pair of randomly chosen vertices in $\cG_t$ as $t\to\infty$.
 Since for a finite set of points $x_1,\ldots, x_k \in \R_{+}$, the probability that $x_1,\ldots, x_k \in \Pi_t^{\eta}$ is 0,  it is traditional to use some tools from Palm theory. First note that the the conditional distribution of $\Pi_t^{\eta}$ given $x_1,\ldots, x_k$ is equivalent to the distribution of $\Pi_t^{\eta} \cup \{x_1,\ldots,x_k\}$ which is the (extended) Slivnyak-Mecke theorem if $\Pi_t^{\eta}$ is a Poisson process; see \cite[Theorem 3.3]{moller_etal:2003}. 
To make the paper self-contained, we state the extended Slivnyak-Mecke theorem below.

\begin{theorem}\label{thm:SM}
Let $S\subset \mathbb{R}^d$ and define $\cal N$ as the set of
locally finite subsets of S. Let $X = (x_i)_{i\ge 1} \in \cal N$ be a Poisson point process on $S$ with a
diffuse and locally finite mean measure $\mu$. Then, for any $n\ge 1$ and any measurable function
$h: \mathcal{N}\times S^n\mapsto [0,\infty)$,
\begin{align*}
\E\left[\sum_{\xi_1\neq\cdots\neq \xi_n\in X}h\left(X\setminus\{\xi_1,\ldots,\xi_n\}\right)\right]
&= \int_{S^n}\E\left[h\left(X\setminus\{\xi_1,\ldots,\xi_n\}\right)\right]
\mu(\diff \xi_1)\cdots \mu(\diff \xi_n).
\end{align*}
\end{theorem}

With Theorem~\ref{thm:SM} available, we present in 
the following proposition that the 2-c-degree of two fixed vertices in a the graphex $\cG_t$ is Poisson distributed.

\begin{prop}{\label{prop:cdegree}}
Let $x,y\in\R_+$, 
and $\zeta_{\{i,j\}}$ be a symmetric array of iid Uniform$[0,1]$ random variables. 
Then $C(x, y, \Pi_t^\eta\cup\{x,y\},(\zeta_{\{i,j\}}))$
follows a Poisson distribution with rate $t\mu_2(x,y)$.
\end{prop}
\begin{proof}
In the sequel, we suppress the notation as 
$$C(x,y)\equiv C(x, y, \Pi_t^\eta\cup\{x,y\},(\zeta_{\{i,j\}})),$$ and use
$i(s)\equiv i(s,\Pi_t^\eta)$ to
denote the index of the point $s\in \Pi_t^\eta$ with respect to the natural
ordering on $\mathbb{R}_+$.
Then we have
\[
C(x,y)=\sum_{z\in\Pi_t^\eta} \ind{W(x,z)\geq \zeta_{i(x)i(z)}}\cdot\ind{W(y,z)\geq \zeta_{i(y)i(z)}}.
\]
Note also that
\[
\int_{\R_+}\int_{[0,1]}\int_{[0,1]} \ind{u\leq W(x,z)}\ind{v\leq W(y,z)}t\diff u \diff v \diff z=t\int_{\R_+} W(x,z)W(y,z)\diff z=t\mu_2(x,y)<\infty.
\] 
Then by Campbell's theorem (cf. \cite[Section 5.3]{kingman:1993}), the characteristic function of $C(x,y)$ is
\begin{align*}
    \E[\exp(itC(x,y))] 
     & = \E[\exp(it\sum_{z\in\Pi_t^\eta} \ind{W(x,z)\geq \zeta_{i(x)i(z)}}\cdot\ind{W(y,z)\geq \zeta_{i(y)i(z)}})]\\
     & =\exp\left\{\int_{\R_+}\int_{[0,1]}\int_{[0,1]} (1-e^{it\ind{u\leq W(x,z)}\ind{v\leq W(y,z)}})t\diff u\diff v\diff z\right\}\\
     & =\exp\left\{t\sum_{n=1}^\infty\frac{(it)^n}{n!}\int_{\R_+}\int_{[0,1]}\int_{[0,1]} \ind{u\leq W(x,z)}\ind{v\leq W(y,z)}\diff u\diff v\diff z\right\}\\
     & =\exp\left\{t\mu_2(x,y)(e^{it}-1)\right\}.
\end{align*}
Hence, $C(x,y)$ is a Poisson random variable with rate $t\mu_2(x,y)$.
\end{proof}

Proposition~\ref{prop:cdegree} characterizes the distribution for a fixed pair of $(x,y)$ in $\cG_t$ with $t$ also fixed. In the next section, we extend our analysis to the asymptotic behavior of the 2-c-degree distribution when $t$ goes to infinity.

\subsection{Asymptotic distribution of common connections}
\label{subsec:asy}
\begin{figure}[h]
\centering
\includegraphics[scale=.7]{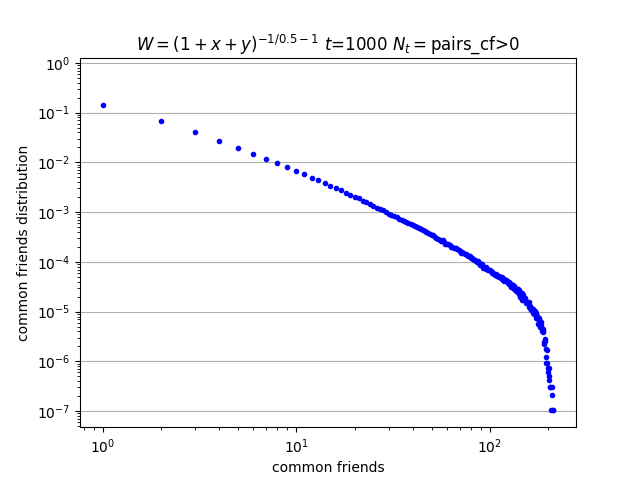}
\caption{The empirical distribution of common connections in a simulated sparse network generated from the graphex process with $W(x,y)=(1+x+y)^{-3}$, $x,y\ge 0$. }
\label{fig:cfrds}
\end{figure}

To analyze the asymptotic distribution, we start with a simulated example where $W(x,y)=(1+x+y)^{-3}$, $x,y\ge 0$. 
We compute the empirical distribution of the number of common friends, i.e., 
\[
\frac{\sum_{x,y\in\Pi_t^\eta}\ind{C(x,y)=k}}{\sum_{x,y\in\Pi_t^\eta}\ind{C(x,y)>0}},\qquad k\ge 1,
\]
and numerical results are plotted in Figure~\ref{fig:cfrds}. We observe that the plot looks linear until a certain point and then it quickly tapers down, indicating a certain cut-off prior to which a power-law behavior is quite evident. Excluding data beyond the cut-off,  we estimate the slope of the log-log plot and obtain an estimated slope between 1 and 2. From this we hypothesize that 
 the 2-c-degree of a randomly chosen pair of vertices, $D_c$, is roughly following
\[
\mathbb{P}(D_c=k)\approx k^{-1-a},\qquad a\in (0,1), 
\]
for large $k$, which implies, for $t$ large,
\begin{align}\label{eq:Dc_approx}
t\mathbb{P}\left(\frac{D_c}{t^{1/a}}>x\right)\approx x^{-a}.
\end{align}
Since $1/a>1$, then for a randomly chosen pair of vertices (as long as we move away from the cutoff), 
 the growth rate of the common friends is faster than $O(t)$. 
Note that by the construction of the model, for $x,y$ small, the expected number of common friends between vertices with labels $x,y$ is $t\mu_2(x,y)$, which is growing no faster than $O(t)$. This leads to the cut-off as observed in the figure.
Later in Theorem~\ref{thm:Cxy}, we stay away from the cut-off, and give the theoretical analysis on the asymptotic features for $x,y \ge b(t)\epsilon$, with $b(t)\to\infty$.

\begin{defn}\label{def:Cxy}
Define $N_{t}^\epsilon(k)$ as the number of vertex pairs with second  co-ordinate value $\eta>b(t)\epsilon$ and who have $k\geq1$ common friends in the graph $\cG_t$, i.e. 
\[
N_{t}^\epsilon(k):=\sum_{x,y\in\Pi_t^\eta,x,y>b(t)\epsilon} \ind{C(x,y)=k}, \qquad \epsilon>0.
\]
\end{defn}

We study the behavior of $N_{t}^\epsilon(k)$ when $W$ is multivariate regularly varying. When 
$W\in\MRV(-\alpha, h(t),\omega(x,y),\R_+^2\setminus\{\boldsymbol{0}\})$ is non-separable, we choose the scaling function to be $b\in \mathcal{RV}_{1/(2\alpha-1)}$, $\alpha>2$, such that 
\begin{align}
\label{eq:bt_nonsep}
\lim_{t\to\infty} t \mu_2(b(t)x, b(t)y) = \int_0^\infty \omega(x,z)\omega(y,z)\diff z.
\end{align}
When $W$ is separable, i.e., $W(x,y)=U(x)U(y)$,  and $W\in \MRV(-2\alpha,U^2(t),\omega(x,y), (0,\infty)^2)$, we set
$b(t) =(1/U)^{\leftarrow}(\sqrt{t})$, so that $b\in \mathcal{RV}_{1/(2\alpha)}$, 
$\alpha>1$ and 
$$
\mu_2\in \MRV\left(-2\alpha,b(t), \int_0^\infty \omega(x,z)\omega(y,z)\diff z, (0,\infty)^2\right).
$$

%
%
Assume that for $q>1/2$ and some constant $C_4>0$, we have
\begin{align}
\mu_4(x_1,x_2,x_3,x_4) &\le C_4 \left(\mu_2(x_1,x_2)\mu_2(x_3,x_4)\right)^q.\label{eq:mu4}
\end{align}
We now give the asymptotic behavior of $\E\left[N_{t}^\epsilon(k)\right]$ for both separable and non-separable cases. 

\begin{theorem}\label{thm:Cxy}
Suppose (i) $W\in \MRV(-\alpha,\R^2_+\setminus\{\bzero\})$ and condition \eqref{W_uni} is satisfied, or, (ii) $W(x,y)=U(x)U(y)$ with $U\in \RV_{-\alpha}$ for $\alpha>0$. Also  assume that \eqref{eq:mu4} is satisfied, 
and that $N_{t}^\epsilon(k)$ is as defined in Definition~\ref{def:Cxy}.
Let $b(t)\in \RV_{1/\gamma}$ for $\gamma>1$ and $\lambda(x,y)$ be some functions satisfying
\begin{align}\label{eq:lim_mu2}
\lim_{t\to\infty}t\mu_2(b(t)x, b(t)y) = \lambda(x,y),\qquad x,y >0.
\end{align}
Then for $k\ge 1$,
 we have
\begin{align}\label{eq:limit_Nk}
\frac{N_{t}^\epsilon(k)}{\left(tb(t)\right)^2}
\convp \frac{1}{k!} \int_\epsilon^\infty\int_\epsilon^\infty \left(\lambda(u,v) \right)^k e^{-\lambda (u,v)}
\diff u\diff v.
\end{align} 

\end{theorem}
\begin{proof}
First note that \eqref{eq:lim_mu2} is a consequence of either Theorem \ref{thm:W_uni} where $\gamma=2\alpha-1$, or, Lemma \ref{lem:Vseparable} where $\gamma=2\alpha$. From Proposition \ref{prop:cdegree}, we have for  $x,y>0$, $C(x,y)$ follows a Poisson distribution with rate $t\mu_2(x,y)$. Now
by the extended Slivnyak-Mecke theorem given in Theorem~\ref{thm:SM}, we see that
\begin{align}
	\E[N_{t}^\epsilon(k)]&=t^2\int_{b(t)\epsilon}^\infty\int_{b(t)\epsilon}^\infty\mathbb{P}[C(x,y)=k]\diff x \diff y\nonumber\\
	&=t^2\int_{b(t)\epsilon}^\infty\int_{b(t)\epsilon}^\infty\frac{(t\mu_2(x,y))^k}{k!}e^{-t\mu_2(x,y)}\diff x \diff y\nonumber\\
	&=t^2b(t)^2\int_{\epsilon}^\infty\int_{\epsilon}^\infty \frac{(t\mu_2(b(t)u,b(t)v))^k}{k!}e^{-t\mu_2(b(t)u,b(t)v)}\diff u \diff v,
	\label{eq:int_Nk}
\end{align}
which, either by using Theorem~\ref{thm:W_uni}, or, Lemma~\ref{lem:Vseparable}, converges to the limit on the 
right hand side of \eqref{eq:limit_Nk} when divided by $t^2b(t)^2$ (as $t\to\infty$). With \eqref{eq:int_Nk} available, we then show the convergence in \eqref{eq:limit_Nk} by proving 
\begin{align}\label{eq:claim}
\text{Var}\left[N_{t}^\epsilon(k)\right] = O\left(t^{-\kappa}\ell(t) \E\left[N_{t}^\epsilon(k)\right]^2\right),
\end{align}
for some $\kappa>0$ and a slowly varying function $\ell(t)$.
Note that 
\begin{align*}
\left(N_{t}^\epsilon(k)\right)^2 &= N_t^\epsilon(k) + \sum_{x_i\in\Pi_t^\eta, i=1,2,3} 
\ind{C(x_1,x_2)=k, C(x_1,x_3)=k}\ind{x_1\neq x_2\neq x_3}\ind{x_i>b(t)\epsilon, i=1,2,3}\\
&\quad+\sum_{x_i\in\Pi_t^\eta, i=1,\ldots,4} 
\ind{C(x_1,x_2)=k, C(x_3,x_4)=k}\ind{x_1\neq x_2\neq x_3\neq x_4}\ind{x_i>b(t)\epsilon,i=1,\ldots,4}\\
&=: N_t^\epsilon(k) + A_t^\epsilon(k) +B_t^\epsilon(k) .
\end{align*}
Since
\[
\mu_4(x_1,x_2,x_3,x_4)\le \frac{1}{2}\left(\mu_2(x_1,x_2)+\mu_2(x_3,x_4)\right),
\]
then applying Theorem~\ref{thm:SM} and the assumption in \eqref{eq:mu4} gives that 
\begin{align}
\E&\left[B_t^\epsilon(k) \right]- \left(\E\left[N_{t}^\epsilon(k)\right]\right)^2 \nonumber\\
&\le t^4\sum_{l=0}^{k-1} \int_{(b(t)\epsilon,\infty)^4}\frac{\left(t\mu_2(x_1,x_2)\right)^l}{l!}
\frac{\left(t\mu_2(x_3,x_4)\right)^l}{l!} \frac{\left(t\mu_4(x_1,x_2,x_3,x_4)\right)^{k-l}}{(k-l)!}\nonumber\\
&\qquad\quad \times e^{-t(\mu_2(x_1,x_2)+\mu_2(x_3,x_4)-\mu_4(x_1,x_2,x_3,x_4))}
\diff x_1 \diff x_2 \diff x_3 \diff x_4\nonumber\\
&\le t^4\sum_{l=0}^{k-1} \int_{(b(t)\epsilon,\infty)^4}\frac{\left(t\mu_2(x_1,x_2)\right)^l}{l!}
\frac{\left(t\mu_2(x_3,x_4)\right)^l}{l!} \frac{\left(t\mu_4(x_1,x_2,x_3,x_4)\right)^{k-l}}{(k-l)!}\nonumber\\
&\qquad\quad \times e^{-t/2(\mu_2(x_1,x_2)+\mu_2(x_3,x_4))}
\diff x_1 \diff x_2 \diff x_3 \diff x_4\nonumber\\
&\le \sum_{l=0}^{k-1} \frac{C_4^{2l} t^{4+k+l}}{(l!)^2 (k-l)!} \left(\int_{(b(t)\epsilon,\infty)^2}
\left(\mu_2(x_1,x_2)\right)^{l+q(k-l)}e^{-t/2\mu_2(x_1,x_2)}
 \diff x_1 \diff x_2 \right)^2,\label{eq:var_upper}\\
\intertext{and using the results in \eqref{eq:int_Nk}, the upper bound above is of order}
&= O\left(t^{4-(2q-1)}(b(t))^4\right) =  O\left(t^{-(2q-1)}(tb(t))^4\right)  .\label{eq:var_bound}
\end{align}

For $\mu_3$, we have that
\begin{align*}
\mu_3(x_1,x_2,x_3)&= \int_0^\infty \prod_{i=1}^3 W(x_i,z)\diff z\\
& = \int_0^\infty \left(\sqrt{W(x_1,z)}W(x_2,z)\right)\left(\sqrt{W(x_1,z)}W(x_3,z)\right)\diff z\\
&\le \left(\int_0^\infty W(x_1,z)W(x_2,z)^2\diff z\right)^{1/2}
\left(\int_0^\infty W(x_1,z)W(x_3,z)^2\diff z\right)^{1/2}\\
&\le \sqrt{\mu_2(x_1,x_2)\mu_2(x_1,x_3)}.
\end{align*}
Then, similarly, for $\E\left[A_t^\epsilon(k) \right]$, the extended Slivnyak-Mecke theorem gives:
\begin{align}
\E\left[A_t^\epsilon(k) \right]
&\le t^3\sum_{l=0}^{k} \int_{(b(t)\epsilon,\infty)^3}\frac{\left(t\mu_2(x_1,x_2)\right)^l}{l!}
\frac{\left(t\mu_2(x_1,x_3)\right)^l}{l!} \frac{\left(t\mu_3(x_1,x_2,x_3)\right)^{k-l}}{(k-l)!}\nonumber\\
&\qquad\quad \times e^{-t(\mu_2(x_1,x_2)+\mu_2(x_1,x_3)-\mu_3(x_1,x_2,x_3))}
\diff x\diff x_1 \diff x_2\nonumber\\
&\le \sum_{l=0}^{k} \frac{t^{3+k+l}}{(l!)^2 (k-l)!} \int_{(b(t)\epsilon,\infty)^3}
\left(\mu_2(x_1,x_2)\right)^{l+(k-l)/2}\left(\mu_2(x_1,x_3)\right)^{l+(k-l)/2}\nonumber\\
&\qquad\quad e^{-t/2(\mu_2(x_1,x_2)+\mu_2(x_1,x_3))}
\diff x_1 \diff x_2 \diff x_3, \nonumber\\
\intertext{and applying the same argument as in \eqref{eq:int_Nk}, the bound is of order}
&= O\left(t^3 \left(b(t)\right)^3\right) = O(t^{-(1+1/\gamma)}\ell(t) (tb(t))^4),\label{eq:var_upper2}
\end{align}
where $b(t)\in \RV_{1/\gamma}$ and we write $1/b(t)= t^{-(1+1/\gamma)}\ell(t)$ for some slowly varying function $\ell(t)$.
Combining \eqref{eq:var_upper2} with \eqref{eq:var_bound},
\eqref{eq:claim} holds with $\kappa \equiv \min(2q-1,1+\frac 1\gamma)>0$.
\end{proof}

\medskip

\begin{remark}
When $W(\cdot,\cdot)$ is separable, the condition in \eqref{eq:mu4} holds with $q=1$.
Then the upper bounds in \eqref{eq:var_upper} and \eqref{eq:var_upper2} are of orders
\[
O\left(t^{2/\alpha+3}\left(\ell_1(t)\right)^4\right),
\qquad O\left(t^{3/(2\alpha)+3}\left(\ell_1(t)\right)^3\right),
\]
respectively, where $\ell_1(t)$ is a slowly varying function such that $b(t)=t^{1/(2\alpha)}\ell_1(t)$.
Therefore, the convergence results in Theorem~\ref{thm:Cxy} hold for $W(\cdot,\cdot)$ separable.
\end{remark}

\medskip

\subsubsection{Tail asymptotics for common connections}
Note that the convergence results in Theorem~\ref{thm:Cxy} do not provide an explicit explanation for the power-law behavior as observed in Figure~\ref{fig:cfrds}. Hence, we proceed with a detailed discussion on the tail distribution of 2-c-degrees.
Consider the function
\[
f(x) = x^k e^{-x}, \qquad x\in A,
\]
and we see that 
$f$ is increasing in $x$ if $k\ge \sup_{x\in A} x$.

\paragraph{Separable case.}
In the separable case, assume in addition that $U(x)U(y)\le U(x+y)$, $x,y\ge 0$, then
the integrand in \eqref{eq:int_Nk} satisfies
\[
\left(t\mu_2(b(t)x,b(t)y)\right)^k e^{-t\mu_2(b(t)x,b(t)y)}
\le \left(C\cdot tU(b(t)(x+y))\right)^k e^{-C\cdot tU(b(t)(x+y))},
\]
with $C=\int_0^\infty U^2(z)\diff z$,
as long as $k\ge\sup_{x,y\ge b(t)\epsilon} t\mu_2(b(t)x,b(t)y)
= \lambda(\epsilon,\epsilon)$.
Therefore, for $k\ge \lambda(\epsilon,\epsilon)$, we have
\begin{align}
\frac{1}{k!}&\int_{\epsilon}^\infty\int_{\epsilon}^\infty
\left(t\mu_2(b(t)x,b(t)y)\right)^k e^{-t\mu_2(b(t)x,b(t)y)} \diff x \diff y\nonumber\\
&\le \frac{1}{k!}\int_{\epsilon}^\infty\int_{\epsilon}^\infty
\left(C\cdot tU(b(t)(x+y))\right)^k e^{-C\cdot tU(b(t)(x+y))} \diff x \diff y,
\label{eq:bound_Cxy}
\end{align}
and since $b(t)=(1/U)^{\leftarrow}(\sqrt{t})$ in the separable case, then
for some slowly varying function $\ell$, we have
\[
tU(b(t)(x+y)) = t^{1/2}(x+y)^{-\alpha}\frac{\ell(b(t)(x+y))}{\ell(b(t))}.
\]
Assume further that
\begin{align}\label{eq:sep_ell}
\lim_{t\to\infty}\ell(t) = C_0>0,
\end{align}
 then
the upper bound in \eqref{eq:bound_Cxy} is of order $O(t^{1/\alpha} k^{-1-2/\alpha})$.
By \eqref{eq:Dc_approx}, we may write $k=O(t^{1/a})$, $a\in (0,1)$, then the upper bound is of order
\[
k^{-\frac{1}{\alpha}(2-a)-1},
\]
and $\frac{1}{\alpha}(2-a)+1\in (1+1/\alpha, 1+2/\alpha)$.

Next, we give a lower bound for the asymptotic power-law behavior of the c-degrees in the separable case.
Since
\begin{align*}
\mu_2(x,y) = CU(x)U(y) \le \frac C2 (U^2(x)+U^2(y)),
\end{align*}
then we have for $b\in \mathcal{RV}_{1/(2\alpha)}$,
\begin{align}
\frac{b^2(t)}{k!}& \int_{\epsilon}^\infty\int_{\epsilon}^\infty
\left(t\mu_2(x,y)\right)^k e^{-t\mu_2(x,y)} \diff x \diff y\nonumber\\
&\ge \frac{1}{k!}\int_{\epsilon}^\infty\int_{\epsilon}^\infty
\left[tCU\bigl(b(t)x\bigr)U\bigl(b(t)y\bigr)\right]^k e^{-t \frac C2 \left(U^2\bigl(b(t)x\bigr)+U^2\bigl(b(t)y\bigr)\right)} \diff x \diff y \nonumber\\
& = C^k \frac{t^{k}}{k!}\left[\int_0^\infty U^k\bigl(b(t)x\bigr) e^{-\frac{tC}2 U^2\bigl(b(t)x\bigr)} \diff x
-\int_0^\epsilon U^k\bigl(b(t)x\bigr) e^{-\frac{tC}2 U^2\bigl(b(t)x\bigr)} \diff x
\right]^2. \nonumber
\end{align}
Using \cite[Lemma S3.5]{caron_etal:2020}, since $U^2(x) \in \RV_{-2\alpha}$,
\begin{align*}
\int_0^\infty U^k(x) e^{-\frac{tC}2 U^2(x)} \diff x \sim \frac{1}{2\alpha}\left(\frac{tC}{2}\right)^{\frac 1{2\alpha}-\frac k2} \Gamma\left(\frac k2 -\frac 1{2\alpha}\right) \ell_1(t), 
\end{align*}
for some slowly varying function $\ell_1(t)$. 
Note that when $\epsilon$ is small, 
\[
\int_0^\epsilon\left(t^{1/2}U\bigl(b(t)x\bigr)\right)^k e^{-\frac{Ct}{2}U^2\bigl(b(t)x\bigr)}\diff x
\approx \int_0^\epsilon x^{-\alpha k}e^{-\frac{C}{2}x^{-2\alpha}}\diff x,
\]
which is also small.
Therefore, using Stirling's approximation,
\begin{align}
\frac{1}{k!}&\int_{b(t)\epsilon}^\infty\int_{b(t)\epsilon}^\infty
\left(t\mu_2(x,y)\right)^k e^{-t\mu_2(x,y)} \diff x \diff y\nonumber\\
& \gtrapprox \;  C^k \frac{t^{k}}{k!}  (2\alpha)^{-2} \left(\frac{tC}2\right)^{\frac 1{\alpha}-k} \left(\Gamma\left(\frac k2 -\frac 1{2\alpha}\right)\right)^2 (\ell_1(t))^2 \nonumber\\
& = \frac{C^{1/\alpha}}{2^{2+1/\alpha-k}\alpha^2} \cdot t^{\frac1{\alpha}} \cdot  \left(\frac{\Gamma(\frac k2 - \frac 1{2\alpha})}{\Gamma(\frac k2+1)} \right)^2  \cdot \left(\frac{\Gamma(\frac k2+1)} {\Gamma{(k+1)}}\right)^2 (\ell_1(t))^2 \nonumber\\
& \approx \frac{C^{1/\alpha}}{2^{2+1/\alpha-k}\alpha^2} \cdot t^{2+\frac1{\alpha}} \cdot \left(\frac{k}{2}\right)^{-\frac 1\alpha-2} \cdot \frac{\sqrt{2\pi k}}{2^{k+1}} \; (\ell_1(t))^2 \nonumber\\
& = O\left(t^{\frac 1\alpha} k^{-\left(\frac 1\alpha+\frac32\right)}\right),\label{eq:lower}
\end{align}
if we assume \eqref{eq:sep_ell}.  Hence for small $\epsilon$ and $k$ is large, \eqref{eq:lower} provides a lower bound for the asymptotic power-law behavior of the 2-c-degree distribution, i.e., for large $t$, 
\[ \frac{N_t^{\epsilon}(k)}{t^2b(t)^2} \gtrapprox \; O\left(k^{-\left(\frac 1\alpha+\frac32\right)}\right) \quad\quad (\text{in probability}).\]


\paragraph{Non-separable case.}
In the non-separable case, we need 
the following technical assumption in order to further study the limit behavior of 
$\E[N_{t}^\epsilon(k)]/(tb(t))^2$.
For $(x,y)\in(0,\infty)^2$, assume 
there exist constants $C_0 > 0$ and $x_0 \geq 0$ such that for all $ x, y \geq x_0$,
\begin{align}
W(x, y) \leq C_0(\mu_1(x))^\theta(\mu_1(y))^\theta, \quad \mu_1(x_0) > 0,
\label{ass2}
\end{align}
where $\theta>1/2$ is a positive constant.
As discussed in \cite{caron_etal:2020}, the assumption in \eqref{ass2} is satisfied with $\theta=0$ when the function $W$ is separable. 
Another example is 
\begin{align}\label{eq:ex_nonsep}
W(x,y) = (1+x+y)^{-\alpha},\qquad\alpha>2,
\end{align}
 then we have
\begin{align*}
\mu_1(x) = \frac{1}{\alpha-1}(1+x)^{-(\alpha-1)},
\qquad
W(x,y) \le \left({\alpha-1}\right)^{\frac{\alpha}{\alpha-1}}\left(\mu_1(x)\mu_1(y)\right)^{\frac{\alpha}{2(\alpha-1)}}.
\end{align*}

In fact, the assumption in \eqref{ass2} implies for $x,y>0$,
\begin{align*}
\mu_2(x,y) &=\int_0^\infty W(x,z)W(y,z)\diff z\\
&\le C_0^2 \int_0^\infty \mu_1(z)^{2\theta}\diff z \cdot\left(\mu_1(x)\mu_1(y)\right)^\theta
=: C'\left(\mu_1(x)\mu_1(y)\right)^\theta.
\end{align*}
Let $b\in \mathcal{RV}_{1/(2\alpha-1)}$ be some scaling function such that
\eqref{eq:bt_nonsep} holds. Assume further that $\mu_1(x)\mu_1(y)\le \mu_1(x+y)$ 
then we have for $x,y\ge \epsilon>0$,
\begin{align*}
t\mu_2(b(t)x,b(t)y) \le t\left(\mu_1\bigl(b(t)(x+y)\bigr)\right)^\theta
= t^{1-\frac{\theta(\alpha-1)}{2\alpha-1}} (x+y)^{-(\alpha-1)}\frac{\ell_0(b(t)(x+y))}{\left(\ell_b(t)\right)^{\alpha-1}},
\end{align*}
where $\ell_0(\cdot)$, $\ell_b(\cdot)$ are two slowly varying functions satisfying
\[
\mu_1(x)=x^{-(\alpha-1)}\ell_0(x),\qquad 
b(t) = t^{1/(2\alpha-1)}\ell_b(t).
\]
Suppose also both $\ell_0$ and $\ell_b$ satisfy \eqref{eq:sep_ell}.
Then repeating the reasoning in the separable case, we see that the upper bound for
\begin{align*}
\frac{1}{k!}&\int_{\epsilon}^\infty\int_{\epsilon}^\infty
\left(t\mu_2(b(t)x,b(t)y)\right)^k e^{-t\mu_2(b(t)x,b(t)y)} \diff x \diff y
\end{align*}
is of order
\[
t^{\frac{2}{\theta(\alpha-1)}-\frac{2}{2\alpha-1}} k^{-1-\frac{2}{\theta(\alpha-1)}}
= t^{\frac{4}{\alpha}-\frac{2}{2\alpha-1}}k^{-1-\frac{4}{\alpha}}.
\]
Recall \eqref{eq:Dc_approx}, and we may write $k=O(t^{1/a})$, with $a\in (0,1)$, then the upper bound is of order
\[
k^{-1-\frac{4}{\alpha}(1-a)-\frac{2a}{2\alpha-1}},
\]
and $1+\frac{4}{\alpha}(1-a)+\frac{2a}{2\alpha-1}\in (1+2/(2\alpha-1), 1+4/\alpha)$.
Note that for the example in \eqref{eq:ex_nonsep}, we have $\ell_0, \ell_b\equiv 1$ and 
$\theta = \frac{\alpha}{2(\alpha-1)}$, which satisfy all those conditions listed above.

\section{Simulation Studies}\label{sec:sim}

To see how tight the bounds derived in the previous section are, especially when the entire network is considered (e.g. the scenario in Figure~\ref{fig:cfrds}), we further examine the slope estimates through simulation studies. We first choose two different sets of values for $\alpha$:
\begin{enumerate}
\item[(1)] Separable case: $W(x,y)=(1+x)^{-\alpha}(1+y)^{-\alpha}$, $\alpha\in \{1.5,2,3,4,5\}$.
\item[(2)] Non-separable case: $W(x,y)=(1+x+y)^{-\alpha}$, $\alpha\in \{2, 2.6,3,4,5\}$.
\end{enumerate}
For each combination of $W(\cdot,\cdot)$ and $\alpha$, we simulate 500 replications with $t=1000$.
Here all simulated replications have a similar distribution shape as depicted in Figure~\ref{fig:cfrds}.
In other words, when plotted on a log-log scale, the empirical distribution shows a linearly decaying pattern as long as we stay away from the cutoff. Therefore, we estimate the power-law index by the absolute value of the slope estimate from a simple linear regression. 
To avoid the cutoff area, we choose a proportion of the empirical distribution such that the fitted regression returns an
$R^2\ge 99.5\%$.
This is done by first fitting a linear regression to the entire empirical distribution, then leaving out observations in the cutoff area until the target $R^2$ is achieved.

\begin{figure}[h]
\centering
\includegraphics[scale=.6]{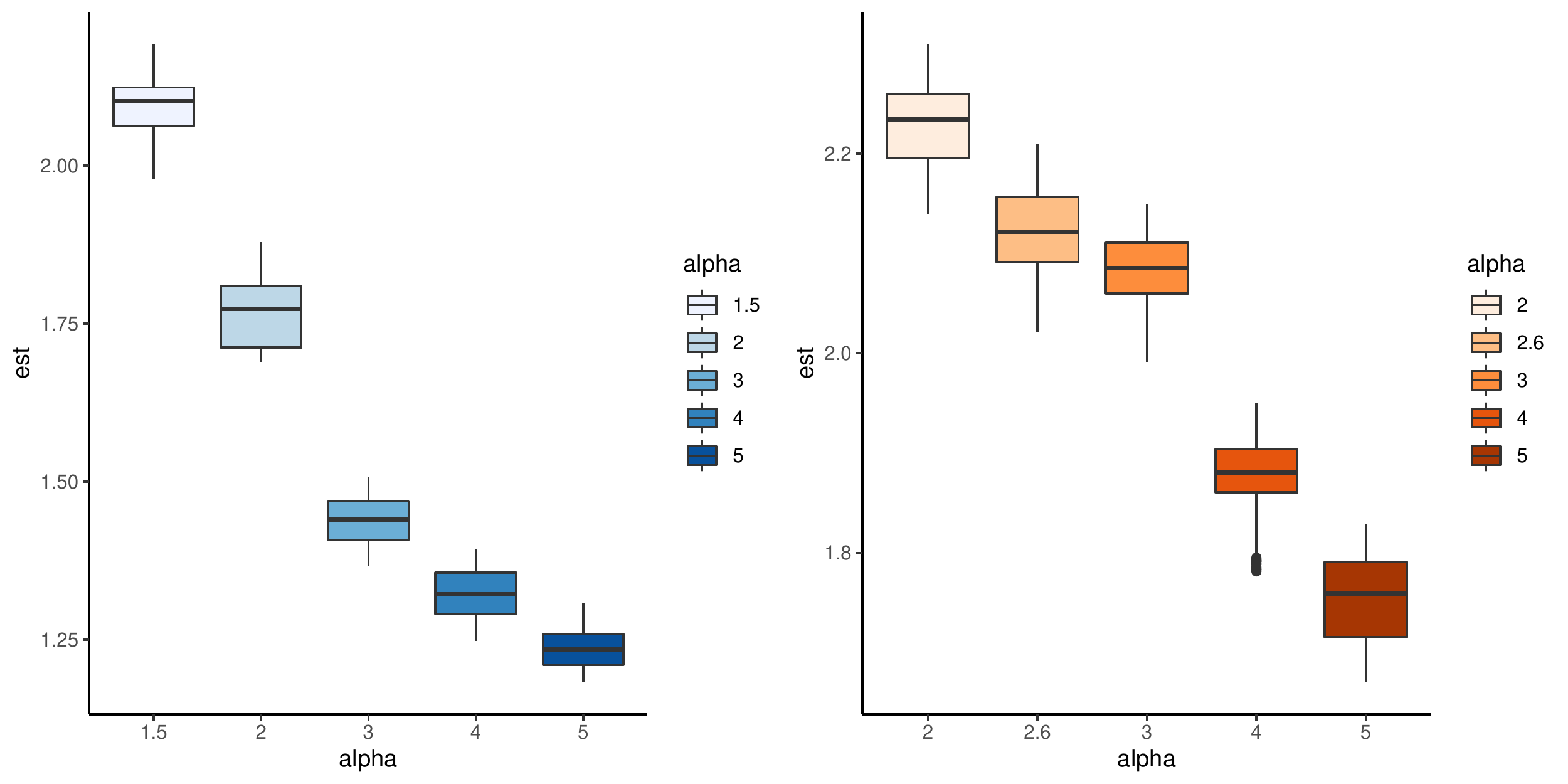}
\caption{Boxplots of estimated tail indices for the distribution of common connections based on simulated sparse networks with graphex function $W(x,y)$. Left: $W(x,y)=(1+x)^{-\alpha}(1+y)^{-\alpha}$, $\alpha\in \{1.5,2,3,4,5\}$. Right: $W(x,y)=(1+x+y)^{-\alpha}$, $\alpha\in \{2, 2.6,3,4,5\}$.}\label{fig:boxplot}
\end{figure}

Following the method described above, we collect all tail index estimates (slope estimates) for the distribution of common connections, and plot them as boxplots in Figure~\ref{fig:boxplot}. The left panel contains all simulated separable examples, i.e. 
$W(x,y)=(1+x)^{-\alpha}(1+y)^{-\alpha}$, $\alpha\in \{1.5,2,3,4,5\}$, whereas all non-separable examples, i.e. $W(x,y)=(1+x+y)^{-\alpha}$, $\alpha\in \{2, 2.6,3,4,5\}$, are given in the right panel of Figure~\ref{fig:boxplot}.
Both separable and non-separable cases reveal the same pattern that when $W(\cdot,\cdot)$ decays at a faster rate or $\alpha$ is large, the distribution of common connections will have a heavier tail.
Similar pattern is also observed from the empirical degree distribution as given in \cite{caron:fox:2017, caron_etal:2020,veitch:roy:2015}.

To assess the tightness of bounds provided in the previous section, we summarize 
the mean, variance and range based on the 500 replications for each combination of $W(\cdot,\cdot)$ and $\alpha$. For the separable case, we compute the \emph{coverage proportion} for the derived upper and lower bounds for the tail index:
\begin{align}
\label{eq:bound_sep}
\left[1+\frac{1}{\alpha},\, \max\left(\frac{3}{2}+\frac{1}{\alpha}, 1+\frac{2}{\alpha}\right)\right],
\end{align}
i.e., the proportion of estimated tail indices that are within the interval in \eqref{eq:bound_sep}. Numerical results are given in Table~\ref{table:sep}.
Overall, the mean and standard deviation are both decreasing as $\alpha$ increases,
and the coverage proportion remains high, as long as $\alpha$ is neither too small nor too large.
This indicates that the bounds indicated in \eqref{eq:bound_sep} perform well in these examples, even though we now take all $(x,y)\in \Pi^\eta_t$ into consideration.

For the non-separable case, we examine the tightness of the derived upper bound for the tail index, i.e.,  
\begin{align}
\label{eq:bound_nonsep}
\left(1+\frac{2}{(2\alpha-1)}, 1+\frac{4}{\alpha}\right),
\end{align}
in Table~\ref{table:nonsep}.
Similar to the separable case, both mean and standard deviation are decreasing as $\alpha$ increases. Except the $\alpha=2$ case, all other coverage proportions reach 100\%. The relative poor coverage proportion for $\alpha=2$ may be due to the fact that $\int_0^\infty\mu_1(x)\diff x=\infty$ when $\alpha=2$. 

Comparing Tables~\ref{table:sep} and \ref{table:nonsep}, we see that for the same $\alpha$, a separable $W(\,\cdot\,,\,\cdot\,)$ tends to generate a heavier tail for the common connection distribution than that in a non-separable $W(\,\cdot\,,\,\cdot\,)$ case. Note that although in both cases $W\in \MRV(-\alpha, \R_+^2\setminus\{\bzero\})$, these two models tend to generate structures which are quite different, for example, we obtain different behavior of graphex marginals $\mu_1$ and $\mu_2$ in the two cases. Using Theorem~\ref{thm:W_uni}, for the non-separable $W$ we have  $\mu_1\in \RV_{-(\alpha-1)}$. On the other hand, for the separable case, $W\in \MRV(-\alpha,\R_+^2\setminus\{\bzero\})$ and moreover $W\in \MRV(-2\alpha,(0,\infty)^2)$, and from Lemma~\ref{lem:Vseparable} we have $\mu_1\in \RV_{-\alpha}$. In order to compare models with similar average connection behaviors, we tabulate in Table~\ref{table:mu1same} the means and ranges of the estimated slope, where the $\mu_1$ tail indices are the same in the two models. Even here we observe that the common connection distribution tends to have heavier tails in the separable case. A similar comparison matching $\mu_2$ indices in the separable ($\alpha=1.5$ with $\mu_2$ index $2\alpha=3$) and non-separable ($\alpha=2$ with $\mu_2$ index $2\alpha-1=3$) cases also leads to observing a heavier tail behavior in the former. 

\begin{table}[h!]
\centering
\begin{tabular}{c| c c c c}
\hline
$\alpha$ & Mean & Std Dev & Range & Coverage (\%)\\
\hline
1.5 & 2.095 & 0.052 & $[1.980, 2.193]$ & 88.8\\
2 & 1.769 & 0.056 & $[1.689, 1.879]$ & 100\\
3 & 1.438 & 0.039 & $[1.366, 1.508]$ & 100\\
4 & 1.322 & 0.041 & $[1.249, 1.394]$ & 99\\
5 & 1.238 & 0.033 & $[1.182, 1.308]$ & 84.8\\
\hline
\end{tabular}
\caption{Mean, standard deviation, and range based on the 500 replications for each combination of $W(\cdot,\cdot)$ and $\alpha$ in the separable case. The coverage proportion of the upper and lower bounds in \eqref{eq:bound_sep} is given in the last column.}\label{table:sep}
\end{table}

\begin{table}[h!]
\centering
\begin{tabular}{c| c c c c}
\hline
$\alpha$ & Mean & Std Dev & Range & Coverage (\%)\\
\hline
2 & 2.228 & 0.045 & $[2.140, 2.310]$ & 81.4\\
2.6 & 2.121 & 0.049 & $[2.022, 2.210]$ & 100\\
3 & 2.081 & 0.041 & $[1.992, 2.150]$ & 100\\
4 & 1.878 & 0.040 & $[1.781, 1.950]$ & 100\\
5 & 1.754 & 0.045 & $[1.670, 1.829]$ & 100\\
\hline
\end{tabular}

\caption{Mean, standard deviation, and range based on the 500 replications for each combination of $W(\cdot,\cdot)$ and $\alpha$ in the non-separable case. The coverage proportion of the upper bound is given in the last column.}\label{table:nonsep}
\end{table}

\begin{table}[h!]
\centering
\begin{tabular}{c| c c c | c c c}
\hline 
& \multicolumn{3}{|c|}{Separable model} & \multicolumn{3}{|c}{Non-separable model}\\
\hline
Index of $\mu_1$ & $\alpha$ & Mean & Range & $\alpha$ & Mean & Range \\
\hline
2 & 2 & 1.769 & $[1.689, 1.879]$   & 3 & 2.081 & $[1.992, 2.150]$\\
3 & 3 & 1.438 & $[1.366, 1.508]$   & 4 & 1.878 & $[1.781, 1.950]$\\
4 & 4 & 1.322 & $[1.249, 1.394]$   & 5  & 1.754 &$[1.670, 1.829]$ \\
\hline
\end{tabular}

\caption{The $\alpha$ value, mean and range based on the 500 replications for each combination of $W(\cdot,\cdot)$ and $\alpha$ in both the non-separable case and separable cases where the tail indices of the univariate graphex marginal $\mu_1$ are the same.}\label{table:mu1same}
\end{table}

\section{Conclusion}\label{sec:conclusion}
In this paper, we work under the paradigm of sparse exchangeable graphs generated by multivariate regularly varying graphex functions and investigate the distributional behavior of common connections between pairs of vertices. We conclude that the asymptotic distribution of number of common connections of randomly chosen pair of vertices, given that the vertices chosen are above an (increasing) threshold also has a power-law behavior under certain regularity conditions. We distinguish between two types of generating graphex functions here, separable and non-separable, each of which leads to different asymptotic tail behavior. We manage to derive bounds for the tail behavior but not the exact rate, and our simulation studies verify the bounds obtained.
 
One future direction would be to find appropriate minimal conditions on the graphex function, so that we obtain tighter bounds for the tail rate of the distribution of number of common connections of randomly chosen pair of vertices. Network scientists have found this measure a good indicator for link-prediction between vertices. We believe understanding 2-c-degree functions, and possibly $k$-c-degree functions will also lead us towards a generating mechanism for networks like Facebook and LinkedIn, where connections are often made through friendship recommendation.

\section*{Data Availability Statement}
The datasets generated during and/or analyzed during the current study are available from the corresponding author on reasonable request.

\bibliographystyle{abbrvnat}
\bibliography{bibfilecc}
\end{document}